\documentclass[12pt,reqno]{amsart}
\usepackage{amscd}
\usepackage{hyperref}
\usepackage{latexsym,amssymb}
\usepackage{enumerate}
\usepackage{color}
\usepackage{amsthm}
\usepackage{amsmath}
\usepackage{graphicx}
\newtheorem{corollary}{Corollary}[section]

\newtheorem{lemma}[corollary]{Lemma}
\newtheorem{proposition}[corollary]{Proposition}
\newtheorem{remark}[corollary]{Remark}

\newtheorem{theorem}[corollary]{Theorem}

\newcommand{\mylabel}[1]{\label{#1}
            \ifx\undefined\stillediting
            \else \fbox{$#1$}\fi }
\newcommand{\BE}{\begin{equation}}

\newcommand{\EEQ}{\end{equation}}
\newcommand{\rfb}[1]{\mbox{\rm
   (\ref{#1})}\ifx\undefined\stillediting\else:\fbox{$#1$}\fi}

\newfont{\Blackboard}{msbm10 scaled 1200}

\newfont{\roma}{cmr10 scaled 1200}

\newcommand{\mm}    {{\hbox{\hskip 0.5pt}}}

\newcommand{\bluff} {{\hbox{\raise 15pt \hbox{\mm}}}}
 \allowdisplaybreaks \textwidth=16cm \textheight=23cm
\makeatletter
\def\section{\@startsection {section}{1}{\z@}{-3.5ex plus -1ex minus
    -.2ex}{2.3ex plus .2ex}{\large\bf}}
\makeatother
\def\be{\begin{equation}}
\def\ee{\end{equation}}

\begin{document}
\thispagestyle{empty}
\title{Null controllability of a population dynamics with interior degeneracy}

\author{Idriss Boutaayamou and Younes Echarroudi}
\thanks{Universit\'{e} Ibn Zohr, Facult\'{e} polydisciplinaire de Ouarzazate, e-mail: dsboutaayamou@gmail.com}
\thanks{D\'epartement de
Math\'ematiques,
Facult\'e des Sciences Semlalia, Laboratoire LMDP, UMMISCO (IRD-UPMC), B. P. 2390 Marrakech 40000, Maroc,
e-mails: yecharroudi@gmail.com}


\begin{abstract}
In this paper, we deal with the null controllability of a population dynamics model with an interior degenerate diffusion. To this end, we proved first a new Carleman estimate for the full adjoint system and afterwards we deduce a suitable observability inequality which will be needed to establish the existence of a control acting on a subset of the space which lead the population to extinction in a finite time.
\end{abstract}

\maketitle

\section{Introduction}\label{first-section}
Consider the system
\begin{eqnarray}\label{405}
  {{\partial y} \over {\partial t}} + {{\partial y} \over {\partial a}}-(k(x)y_{x})_{x}+\mu(t,a, x)y = \vartheta\chi_{\omega}  && \text{ in } Q,\\
\nonumber  y(t,a, 1)=y(t,a, 0)=0  && \text{ on }(0,T)\times(0,A),\\
\nonumber y(0, a, x)=y_{0}(a, x)  && \text{ in }Q_{A},\\
\nonumber y(t, 0, x)=\int_{0}^{A} \beta (t,a, x)y (t,a, x)da  &&  \text{ in } Q_{T},
\end{eqnarray}
where $Q=(0,T)\times(0,A)\times(0,1)$, $Q_{A}=(0,A)\times(0,1)$, $Q_{T}=(0,T)\times(0,1)$ and we will denote  $q=(0,T)\times(0,A)\times \omega,$ where $\omega=(x_1, x_2) \subset\subset (0,1)$ is the region  where the control $\vartheta$ is acting. This control corresponds to an external supply or to removal of individuals on the subdomain $\omega$. Since the system \eqref{405} models the dispersion of gene of a given population, then $x$ represents the gene type and $y(t,a,x)$ is the distribution of individuals of age $a$ at time $t$ and of gene type $x$. The parameters $\beta(t,a,x)$ and $\mu(t,a,x)$ are respectively the natural fertility and mortality rates, $A$ is the maximal age of life and $k$ is the gene dispersion coefficient. $y_{0}\in L^{2}(Q_{A})$ is the initial distribution of population. Finally, $\int_{0}^{A} \beta (t,a, x)y (t,a, x)da $ is the distribution of the newborns of the population that are of gene type $x$ at time $t$. As usual, we will suppose that no individual reaches the maximal age $A$. We note that in the most works concerned with the diffusion population dynamics models, $x$ is viewed as the space variable.\\ The population dynamics models in their different aspects attracted many authors and were investigated from many sides (see for example \cite{genni7, genni6, genni4, genni5, Pia, Poz}).
Among those questions, we find the null controllability problem or in general the controllability problems for age and space structured population dynamics models which were studied in a intensive literature basing, in general, on the references interested on the controllability of heat equation (see for instance \cite{Can2, Can3, Can4, Can5, Can6, genni3} for a different controllability problems of heat equation). In this context, we can cite the pioneering items of V. Barbu and al. in \cite{Marcheva}, B. Ainseba and S. Anita in \cite{Ain4, Ain3, Ain2, Ain1}. In \cite{Marcheva}, the authors proved the null controllability for a population dynamics model without diffusion both in the cases of migration and birth control for $T\geq A$ showing directly an appropriate observability inequality for the associated adjoint system and they concluded  that in the case of the migration control, only a classes of age was controlled in contrary with the birth control which allows to steer all population to extinction. In \cite{Ain4, Ain3, Ain2, Ain1}, the diffusion was taken into account in a age-space structured model and the null controllability of \eqref{405} for a classes of age was established in the case where $k=1$ and for any dimension $n$ by means of a weighted estimates called Carleman estimates and exploiting the results gotten for heat equation in \cite{Fursikov}. In \cite{ech}, B. Ainseba and al. studied a more general case allowing the dispersion coefficient to depend on the variable $x$ and verifies $k(0)=0$ (i.e, the coefficient of dispersion $k$ degenerates at 0). The authors tried to obtain \eqref{null-contr} in such a situation with $\beta \in L^{\infty}$ basing on the work done in \cite{Bouss} for the degenerate heat equation to establish a new Carleman estimate for the full adjoint system \eqref{adj-sys} and afterwards his observability inequality. However, the null controllability property of this paper was showed under the condition $T\geq A$ (as in \cite{Marcheva}) and this constitutes a restrictiveness on the "optimality" of the control time $T$ since it means, for example, that for a pest population whose the maximal age $A$ may equal to a many days (may be many months or years) we need much time to bring the population to the zero equilibrium. In the same trend and to overcome the condition $T\geq A$, L. Maniar et al in \cite{man} suggested the fixed point technique in which the birth rate $\beta$ must be in $C^{2}(Q)$ specially in the proof of \cite[Proposition 4.2]{man}. Such a technique consists briefly to demonstrate in a first time the null controllability for an intermediate system with a fertility function $b\in L^{2}(Q_{T})$ instead of $ \int_{0}^{A} \beta (t,a, x)y (t,a, x)da $ and to achieve the task via a Leray-Schauder Theorem.
\\Thereby, the main goal of the current paper is to deal with the null controllability property with a minimum of regularity of $\beta$ (see \eqref{hyp-beta-mu}) and a positive small control time $T$ taking into account that $k$ depends on the gene type and degenerates at a point $x_{0}\in\omega$, i.e $k(x_{0})=0$, e.g $k(x)=|x-x_{0}|^{\alpha}$. To be more accurate, for a fixed $T\in (0, \delta)$ with $\delta\in(0,A)$ small enough, we investigate the existence of a suitable control $\vartheta\in L^{2}(q)$ which depends on $y_{0}$ and $\delta$ and such that the associated solution $y$ of \eqref{405} satisfies
\begin{equation}\label{null-contr}
y(T,a, x)=0, \quad  \text{ a.e. in } (\delta, A)\times(0,1).
\end{equation}
If $k(x_{0})=0$ in a point $x_{0}\in\omega$, we say that \eqref{405} is a population dynamics model with interior degeneracy. Genetically speaking, the meaning of $k(x_{0})=0$ is that the gene of type $x_{0}\in(0,1)$ can not be transmitted from the studied population to its offspring. This objective will be attained via the classical procedure following the strategy of \cite{genni}. On other words, we will establish an appropriate observability inequality for the full adjoint system of \eqref{405} which is an outcome of a suitable Carleman estimate. We highlight that such a result can be shown if we replace the homogeneous Dirichlet boundary conditions by the ones of Neumann, i.e $y_{x}(t,a,0)=y_{x}(t,a,1)=0, (t,a)\in(0,T)\times(0,A)$ using the same way done in \cite{bout}. Another interesting null controllability problem of \eqref{405} can be elaborate using the work of Fragnelli et al. in \cite{genni2} arising in the case when the potential term admits an interior singularity belonging to gene type domain.\\
The remainder of this paper is organized as follows: in Section \ref{second-section}, we will provide the well-posedness of \eqref{405} and give the proof of the Carleman estimate of its adjoint system. The Section \ref{third-section} will be devoted to the observability inequality and hence we get the null controllability result \eqref{null-contr}. The last section will take the form of an appendix where we will bring out a Caccioppoli's inequality which plays an important role to show the desired Carleman estimate.
\section{Well-posedness and Carleman estimate results}\label{second-section}
\subsection{Well-posedness result}
For this section and for the sequel, we assume that the dispersion coefficient $k$ verifies
\begin{equation}\label{hyp-k}
\left\{
\begin{array}{l}
\exists x_{0} \in (0,1),k \in  C([0,1])\cap C^{1}([0,1]\backslash \{x_{0}\}),\;\;
 k>0 \text{ in }[0,1]\backslash \{x_{0}\} \text{ and } k(x_{0})=0,\\
\exists \gamma \in [0,1):  (x-x_{0})k'(x)\leq\gamma k(x), \;  x \in [0,1]\backslash \{x_{0}\}.
\end{array}
\right.
\end{equation}
It is well-known in the literature of degenerate problems that there exist two kinds of degeneracy namely the weakly degenerate and the strong degenerate problems, in our study we will restrict ourselves to the first one and this fact explains the choice of $\gamma\in [0,1)$ which in fact are associated to the Dirichlet boundary conditions (see \cite[Hypothesis 1.1]{genni}). On the other hand, the last hypothesis on $k$ means in the case of $k(x)=|x-x_{0}|^\alpha$ that $0\leq \alpha<1$.\\ The investigation of \eqref{null-contr} needs also the following assumptions on the natural rates $\beta$ and $\mu$
\begin{equation}\label{hyp-beta-mu}
\begin{cases}
\mu, \beta \in L^{\infty}(Q),\quad
\beta(t,a,x), \mu(t,a,x)\geq0,  \text{ a.e. in }Q,\\
\beta(.,0,.)\equiv0 \hspace{0.25cm}\text{ a.e. in } (0,T)\times(0,1).
\end{cases}
\end{equation}
The last assumption in \eqref{hyp-beta-mu} is natural since the newborns are not fertile. Also, it is worth mentioning to point out that, as in \cite{ech} we do not need to require that $\mu$ satisfies an hypotheses like $\int_{0}^{A}\mu(t-s,A-s,x)ds=+\infty, \quad  (t,x) \in [0,T]\times[0,1]$ since it does not play any role on the well-posedness result and the computations concerning the proofs of our controllability result as well. However, we will suppose that no individual can reach the maximal age $A$ as mentioned in the introduction. In the same context, we emphasize that in \cite{man}, the $L^{\infty}-$regularity of $\beta$ is sufficient to prove the well posedness of the studied model which is exactly our case. To this end, we introduce the following weighted Sobolev spaces:
\begin{equation*}\label{Sob-spaces}
\left\{
\begin{array}{l}
H^{1}_{k}(0, 1):=\{u \in L^{2}(0,1):u \text{ is abs.  cont. in } [0,1]:  \sqrt{k}u_{x} \in L^{2}(0,1), u(1)=u(0)=0\},\\
H^2_{k} (0, 1):=\Big\{ u \in H^1_{k}(0, 1) \,: \, k(x)u_x
\in H^1(0, 1)\Big\},
\end{array}
\right.
\end{equation*}
endowed respectively with the norms
\begin{equation*}\label{norme}
\left\{
\begin{array}{l}
\|u\|^{2}_{H^{1}_{k}(0, 1)}
:=\|u\|^{2}_{L^{2}(0,1)}+ \|\sqrt{k}u_{x}\|^{2}_{L^{2}(0,1)}, \quad u \in H^{1}_{k}(0, 1),\\
\|u\|^2_{H^2_{k}} := \|u\|^2 _{H^1_{k}(0, 1)}
+ \|(k(x)u_x)_x\|^2_{ L^2(0,1)}, \quad u \in H^{2}_{k}(0, 1).
\end{array}
\right.
\end{equation*}
We recall from
\cite[Theorem 2.2]{genni1} that the  operator $(P,D(P))$ defined  by
$
Pu := (k(x)u_x)_x,\, \ u \in
D(P) = H^2_{k}(0, 1),
$
is closed
self-adjoint and  negative  with dense domain in $L^2(0, 1)$. Consequently, from \cite[Theorem 5]{web} the operator $\mathcal{A}:=-{{\partial } \over {\partial a}}+P$ generates a $C_0$-semigroup on the space  $L^{2}((0, A)\times(0, 1))$. Then, the following well-posedness result holds.
\begin{theorem}\label{410}
Under the assumptions \eqref{hyp-k} and \eqref{hyp-beta-mu} and for all $\vartheta\ \in L^{2}(Q)$ and $y_{0} \in L^{2}(Q_{A})$, the system \eqref{405} admits a unique solution $y$. This solution belongs to $E:=C([0, T], L^{2}((0, A)\times(0, 1)))\cap C([0, A], L^{2}((0, T)\times(0, 1))) \cap L^{2}((0, T)\times(0, A), H^{1}_{k}(0, 1))$.
Moreover, the solution of \eqref{405} satisfies the following inequality
\begin{eqnarray*}\label{well-posedness}
\nonumber & &
\sup_{t\in [0,T]}\|y(t)\|^{2}_{L^{2}(Q_{A})}+ \sup_{a\in [0,A]}\|y(a)\|^{2}_{L^{2}(Q_{T})}+\int_{0}^1\int_{0}^{A}\int_{0}^{T}
(\sqrt{k(x)}y_x)^2dtdadx\\ & &\leq C\left(\int_{q}\vartheta^{2}+
\int_{Q_{A}}y_{0}^{2}dadx\right).
\end{eqnarray*}
\end{theorem}
\subsection{Carleman estimates results}
As we said in the introduction, we will show the main key of this paper namely the Carleman type inequality. In general, it is well-known that to prove a controllability result of a studied model  through this a priori estimate, we must show this last for the associated adjoint system. In our case, this adjoint system takes the following form
\begin{align}
&{{\partial w} \over {\partial t}} + {{\partial w} \over {\partial a}}+(k(x)w_{x})_{x}-\mu(t,a, x)w =-\beta(t,a,x)w(t,0,x),\nonumber\\
& w(t,a, 1)=w(t,a, 0)=0, \label{adj-sys}\\
&\nonumber  w(T,a, x)=w_{T}(a, x),\\
&\nonumber  w(t,A, x)=0,
\end{align}
where  $T>0$ and  assume that $w_{T} \in L^{2}(Q_{A})$. Of course, the assumptions \eqref{hyp-k} and \eqref{hyp-beta-mu} on $k$, $\mu$ and $\beta$ are perpetuated. To attaint our goal, we will introduce the following weight functions
\begin{equation}\label{weight-functions}
\left\{
\begin{array}{l}
\varphi(t,a, x):=\Theta  (t, a)\psi(x),
\\
\Theta  (t, a):= \displaystyle\frac{1}{(t(T-t))^{4}a^{4}},\\\\
\psi(x):= c_{1}\left(\int_{x_{0}}^{x}\frac{r-x_{0}}{k(r)} dr-c_{2}\right).
\end{array}
\right.
\end{equation}
For the moment, we will assume that $c_{2}>\max\{\frac{(1-x_{0})^{2}}{k(1)(2-\gamma)}, \frac{x_{0}^{2}}{k(0)(2-\gamma)} \}$ and $c_{1}>0$. A more precise restriction on $c_{1}$ will be given later. On the other hand, using the relation satisfied by $c_{2}$ and with the aid of \cite[Lemma 2.1]{genni} one can prove that $\psi(x)<0 \hspace{0.25cm} \forall x\in [0,1]$. Observe also that $\Theta(a,t)\rightarrow +\infty$ as $t\rightarrow T^{-}, 0^{+}$ and $a\rightarrow 0^{+}.$\\ To demonstrate our Carleman estimate, we require that $k$ fulfills, besides \eqref{hyp-k} the following hypothesis
\begin{equation}\label{H-P}
\left\{
\begin{array}{l}
\exists \theta\in (0, \gamma] \text{ such that :}\\
x\mapsto\frac{k(x)}{|x-x_{0}|^{\theta}} \text{ is nonincreasing on the left of } x=x_{0}\\
\text{ and nondecreasing on the right of } x=x_{0},
\end{array}
\right.
\end{equation}
where $\gamma$ is defined by \eqref{hyp-k}. The first Carleman estimate result is the following
\begin{proposition}\label{Intermediate-Carleman estimate}
Consider the two following systems with $h \in L^{2}(Q)$
 \begin{eqnarray}\label{inter-adjoint sys-1}
{{\partial w} \over {\partial t}} + {{\partial w} \over {\partial a}}+(k(x)w_{x})_{x} &=& h,\\
\nonumber  w(a, t, 1)=w(a, t, 0)&=&0, \\
\nonumber  w(T, a, x)&=&w_{T}(a, x),\\
\nonumber  w(t, A, x)&=&0,
\end{eqnarray}
and
\begin{eqnarray}\label{inter-adjoint sys-2}
 {{\partial w} \over {\partial t}} + {{\partial w} \over {\partial a}}+(k(x)w_{x})_{x}-\mu(t, a, x)w &=& h,\\
\nonumber  w(t, a, 1)=w(t, a, 0)&=&0, \\
\nonumber  w(T, a, x)&=&w_{T}(a, x),\\
\nonumber  w(t, A, x)&=&0.
\end{eqnarray}
 Then, there exist two positive constants
$C$ and $s_{0}$, such that every solutions of \eqref{inter-adjoint sys-1} and  \eqref{inter-adjoint sys-2} satisfy, for all $s\geq s_{0}$, the following inequality
\begin{eqnarray}\label{Inter-Carl}
& &s^{3}\int_{Q}  \Theta^{3}\frac{(x-x_{0})^{2}}{k(x)}w^{2}e^{2s\varphi} dtdadx + s\int_{Q}  \Theta k(x)w_{x}^{2}e^{2s\varphi} dtdadx \\ \nonumber& &\leq C\left(\int_{Q}\mid h\mid^{2}e^{2s\varphi} dtdadx+s \int_{0}^{A}\int_{0}^{T}[k\Theta e^{2s\varphi} (x-x_{0})w_{x}^{2}]_{x=0}^{x=1} dtda\right).
\end{eqnarray}
\end{proposition}
\proof
Firstly, we will prove \eqref{Inter-Carl} for system \eqref{inter-adjoint sys-1} and replacing $h$ by $h+\mu w$ we will get the same inequality for \eqref{inter-adjoint sys-2}. So, let $w$ be the solution of \eqref{inter-adjoint sys-1} and put $$\nu(t, a, x):=e^{s\varphi(t, a, x)}w(t, a, x)$$. Then, $\nu$ satisfies the following system
\begin{eqnarray}\label{Carl-1}
L^{+}_{s}\nu +  L^{-}_{s}\nu &=& e^{s\varphi(t, a, x)}h, \\
\nonumber  \nu(t, a, 1)=\nu(t, a, 0)&=&0, \\
\nonumber  \nu(T, a, x)=\nu(0, a, x)&=&0,\\
\nonumber  \nu(t, A, x)=\nu(t, 0, x)&=&0,
\end{eqnarray}
 where $$L^{+}_{s}\nu:=(k(x)\nu_{x})_{x}-s(\varphi_{a}+\varphi_{t})\nu+s^{2}\varphi^{2}_{x}k(x)\nu,$$
and $$L^{-}_{s}\nu:=\nu_{t}+\nu_{a}-2sk(x)\varphi_{x}\nu_{x}-s(k(x)\varphi_{x})_{x}\nu.$$
Passing to the norm in \eqref{Carl-1}, one has
$$\|L^{+}_{s}\nu\|^{2}_{L^{2}(Q)} +  \|L^{-}_{s}\nu\|^{2}_{L^{2}(Q)} + 2<L^{+}_{s}\nu,L^{-}_{s}\nu>= \|e^{s\varphi(a, t, x)}h\|^{2}_{L^{2}(Q)},$$
where $<.,.>$ denotes here the inner product in $L^{2}(Q)$.
Then, the proof of step one is based on the calculus of the inner product $<L^{+}_{s}\nu,L^{-}_{s}\nu>$ whose a first expression is given in the following lemma.
\begin{lemma}\label{Inter-Carl-1}
The following identity holds
$$<L^{+}_{s}\nu,L^{-}_{s}\nu>=S_{1}+S_{2},$$ with
\begin{align}\label{20}
\nonumber S_{1}&=s\int_{Q}(k(x)\nu_{x})^{2}\varphi_{xx}dtdadx-s^{3}\int_{Q}(k(x)\varphi_{x})_{x}k(x)\varphi^{2}_{x}\nu^{2}dtdadx+s^{2}\int_{Q}(\varphi_{a}+\varphi_{t})(k(x)\varphi_{x})_{x}\nu^{2}dtdadx \\ \nonumber &
+s\int_{Q}k(x)\nu_{x}((k(x)\varphi_{x})_{xx}\nu+(k(x)\varphi_{x})_{x}\nu_{x})dtdadx+s^{3}\int_{Q}(k^{2}\varphi^{3}_{x})_{x}\nu^{2}dtdadx \\ &\nonumber-s^{2}\int_{Q}(k(x)(\varphi_{a}+\varphi_{t})\varphi_{x})_{x}\nu^{2}dtdadx+\frac{s}{2}\int_{Q}(\varphi_{at}+\varphi_{tt})\nu^{2}dtdadx-\frac{s^{2}}{2}\int_{Q}(\varphi^{2}_{x})_{t}k(x)\nu^{2}dtdadx\\
&\nonumber+ \frac{s}{2}\int_{Q}(\varphi_{at}+\varphi_{aa})\nu^{2}dtdadx-\frac{s^{2}}{2}\int_{Q}(\varphi^{2}_{x})_{a}k(x)\nu^{2}dtdadx,
\end{align}
and
\begin{eqnarray}\label{21}
\nonumber
S_{2}&=&\int_{0}^{A}\int_{0}^{T}[k(x)\nu_{x}\nu_{a}]_{0}^{1}dtda+\int_{0}^{A}\int_{0}^{T}[k(x)\nu_{x}\nu_{t}]_{0}^{1}dtda \\
&&\nonumber+ s^{2}\int_{0}^{A}\int_{0}^{T}[k(x)\varphi_{x}(\varphi_{a}+\varphi_{t})\nu^{2}]_{0}^{1}dtda-s^{3}\int_{0}^{A}\int_{0}^{T}[k^{2}(x)\varphi^{3}_{x}\nu^{2}]_{0}^{1}dtda \\ &&\nonumber -s\int_{0}^{A}\int_{0}^{T}[k(x)\nu\nu_{x}(k(x)\varphi_{x})_{x}]_{0}^{1}dtda-s\int_{0}^{A}\int_{0}^{T}[(k(x)\nu_{x})^{2}\varphi_{x}]_{0}^{1}dtda.
\end{eqnarray}
\end{lemma}
For the proof of Lemma \ref{Inter-Carl-1}, see the one of \cite[Lemma 3.2]{man}. The previous expressions of $S_{1}$ and $S_{2}$ can be simplified using of the functions $\varphi$ and $\psi$ given in \eqref{weight-functions} and also the homogeneous Dirichlet boundary conditions satisfied by $\nu$. Hence, one has
\begin{eqnarray}\label{S1}
& &
S_{1}=\frac{s}{2}\int_{Q}(\Theta_{aa}+\Theta_{tt})\psi\nu^{2}dxdtda+s\int_{Q}\Theta_{ta}\psi\nu^{2}dtdadx \nonumber\\& &+sc_{1}\int_{Q}\Theta
(2k(x)-(x-x_{0})k'(x))\nu_{x}^{2}dtdadx-2s^{2}\int_{Q}\Theta c_{1}^{2}\frac{(x-x_{0})^{2}}{k(x)}(\Theta_{a}+\Theta_{t})\nu^{2}dtdadx\nonumber\\& &+
s^{3}\int_{Q}\Theta^{3}c_{1}^{3}\left(\frac{x-x_{0}}{k(x)})^{2}(2k(x)-(x-x_{0})k'(x)\right)\nu^{2}dtdadx,
\end{eqnarray}
and
$$S_{2}=-sc_{1}\int_{0}^{A}\int_{0}^{T}[k\Theta e^{2s\varphi} (x-x_{0})\nu_{x}^{2}]_{x=0}^{x=1} dtda.$$
Accordingly,
\begin{eqnarray}\label{scalar-product}
\nonumber
<L^{+}_{s}\nu,L^{-}_{s}\nu>&=&\frac{s}{2}\int_{Q}(\Theta_{aa}+\Theta_{tt})\psi\nu^{2}dxdtda+s\int_{Q}\Theta_{ta}\psi\nu^{2}dtdadx \nonumber\\& &+sc_{1}\int_{Q}\Theta
(2k(x)-(x-x_{0})k'(x))\nu_{x}^{2}dtdadx-2s^{2}\int_{Q}\Theta c_{1}^{2}\frac{(x-x_{0})^{2}}{k(x)}(\Theta_{a}+\Theta_{t})\nu^{2}dtdadx\nonumber\\& &+
s^{3}\int_{Q}\Theta^{3}c_{1}^{3}\left(\frac{x-x_{0}}{k(x)}\right)^{2}(2k(x)-(x-x_{0})k'(x))\nu^{2}dtdadx
\\ \nonumber& &-sc_{1}\int_{0}^{A}\int_{0}^{T}[k\Theta e^{2s\varphi} (x-x_{0})w_{x}^{2}]_{x=0}^{x=1} dtda.
\end{eqnarray}
Thanks to the third assumption in \eqref{hyp-k}, we have
\begin{eqnarray}\label{S1-1}
S_{1}\geq\frac{s}{2}\int_{Q}(\Theta_{aa}+\Theta_{tt})\psi\nu^{2}dtdadx+
s\int_{Q}\Theta_{ta}\psi\nu^{2}dtdadx +sc_{1}\int_{Q}\Theta
k(x)\nu_{x}^{2}dtdadx\\-2s^{2}\int_{Q}\Theta c_{1}^{2}\frac{(x-x_{0})^{2}}{k(x)}(\Theta_{a}+\Theta_{t})\nu^{2}dtdadx  \nonumber+
s^{3}\int_{Q}\Theta^{3}c_{1}^{3}\frac{(x-x_{0})^{2}}{k(x)}\nu^{2}dtdadx.
\end{eqnarray}
Observe that $|\Theta(\Theta_{a}+\Theta_{t})|\leq c\Theta^{3},$ to infer, for $s$ quite large that
\begin{eqnarray}\label{S1-2}
& &\nonumber
\left|-2s^{2}\int_{Q}\Theta c_{1}^{2}\frac{(x-x_{0})^{2}}{k(x)}(\Theta_{a}+\Theta_{t})\nu^{2}dtdadx \right|\\ & & \leq 2s^{2}c_{1}^{2}c\int_{Q}\frac{(x-x_{0})^{2}}{k(x)}\Theta^{3}\nu^{2}dtdadx\leq \frac{c_{1}^{3}}{4}s^{3}\int_{Q}\frac{(x-x_{0})^{2}}{k(x)}\Theta^{3}\nu^{2}dtdadx.
\end{eqnarray}
On the other hand, we have $r\mapsto\frac{|r-x_{0}|^{\gamma}}{k(r)}$ is nondecreasing in the right of $x_{0}$.
\\ Then,
\begin{eqnarray}\label{S1-3}
\nonumber & & |\psi(x)|=|c_{1}l(x)-c_{1}c_{2}|\leq c_{1}\left|\int_{x_{0}}^{x}\frac{r-x_{0}}{k(r)}dr\right|+c_{1}c_{2}\leq c_{1}c_{2}+c_{1}\frac{(1-x_{0})^{2}}{k(1)(2-\gamma)}\\ \nonumber &&\leq \frac{c_{1}}{(2-\gamma)k(1)}+c_{1}c_{2}.
\end{eqnarray}
A simple computations allow us to check that $\Theta_{aa}+\Theta_{tt}+|\Theta_{ta}|\leq C_{1} \Theta^{\frac{3}{2}}$. This yields
\begin{eqnarray}\label{S1-4}
\nonumber& &
\left|\frac{s}{2}\int_{Q}(\Theta_{aa}+\Theta_{tt})\psi\nu^{2}dtdadx+s\int_{Q}\Theta_{ta}\psi\nu^{2}dtdadx\right| \\ \nonumber& &\leq s\left(\frac{c_{1}}{(2-\gamma)k(1)}+c_{1}c_{2}\right)\int_{Q}\left(\frac{\Theta_{aa}+\Theta_{tt}}{2}+|\Theta_{ta}|\right)\nu^{2}dtdadx\\ & &\leq M s\left(\frac{c_{1}}{(2-\gamma)k(1)}+c_{1}c_{2}\right)\int_{Q}\Theta^{\frac{3}{2}}\nu^{2}dtdadx.
\end{eqnarray}
It remains now to bound the term $|\int_{Q}\Theta^{\frac{3}{2}}\nu^{2}dtdadx|$. Using the generalized Young inequality we obtain
\begin{eqnarray}\label{S1-5}
\nonumber &&\left|\int_{0}^{1}\Theta^{\frac{3}{2}}\nu^{2}dx\right|\\ \nonumber && = \left|\int_{0}^{1}\left(\Theta\frac{k^{\frac{1}{3}}}{|x-x_{0}|^{\frac{2}{3}}}\nu^{2}\right)^{\frac{3}{4}}
\left(\Theta^{3}\frac{|x-x_{0}|^{2}}{k}\nu^{2}\right)^{\frac{1}{4}}dx\right|\\ &&
\leq\frac{3\epsilon}{4}\int_{0}^{1}\Theta\frac{k^{\frac{1}{3}}}{|x-x_{0}|^{\frac{2}{3}}}\nu^{2}dx
+\frac{1}{4\epsilon}\int_{0}^{1}\Theta^{3}\frac{|x-x_{0}|^{2}}{k}\nu^{2}dx
\end{eqnarray}
Put
\begin{eqnarray}\label{p}
p(x)=(k(x)|x-x_{0}|^{4})^{\frac{1}{3}}.
\end{eqnarray}
By hypothesis \eqref{H-P}, one can check that $x\mapsto\frac{p(x)}{|x-x_{0}|^{q}}$, with $q:=\frac{4+\theta}{3}\in (1,2)$ is nonincreasing on the left of $x=x_{0}$ and nondecreasing on the right of $x=x_{0}$. Furthermore, we have $\frac{k^{\frac{1}{3}}}{|x-x_{0}|^{\frac{2}{3}}}=\frac{p(x)}{(x-x_{0})^{2}}$ and there exists $C_{2}>0$ such that $p(x)<C_{2}k(x)$.
\\ Hence, by mean of Hardy-Poincar\'{e} inequality (see \cite[Proposition 2.3]{genni}), we conclude that
\begin{eqnarray}\label{S1-6}
\nonumber \int_{0}^{1}\Theta\frac{k^{\frac{1}{3}}}{|x-x_{0}|^{\frac{2}{3}}}\nu^{2}dx
&=&\int_{0}^{1}\Theta\frac{p(x)}{(x-x_{0})^{2}}\nu^{2}dx\\ \nonumber &&\leq C\int_{0}^{1}\Theta p \nu_{x}^{2}dx
\\&& \leq C C_{2}\int_{0}^{1}\Theta k \nu_{x}^{2}dx,
\end{eqnarray}
where $C>0$ is the constant of Hardy-Poincar\'{e}. Combining \eqref{S1-5} and \eqref{S1-6}, we get
\begin{eqnarray}\label{S1-7}
\left|\int_{0}^{1}\Theta^{\frac{3}{2}}\nu^{2}dx\right|\leq \frac{3\epsilon}{4}C_{3}\int_{Q}\Theta k \nu_{x}^{2}dtdadx
+\frac{1}{4\epsilon}\int_{Q}\Theta^{3}\frac{|x-x_{0}|^{2}}{k}\nu^{2}dtdadx.
\end{eqnarray}
Hence, \eqref{S1-4} and \eqref{S1-7} lead to
\begin{eqnarray}\label{S1-8}
\nonumber&&
\left|\frac{s}{2}\int_{Q}(\Theta_{aa}+\Theta_{tt})\psi\nu^{2}dtdadx+s\int_{Q}\Theta_{ta}\psi\nu^{2}dtdadx\right| \\ & &\leq sc_{1}C_{4}\epsilon\int_{Q}\Theta k \nu_{x}^{2}dtdadx
+\frac{sc_{1}C_{5}}{4\epsilon}\int_{Q}\Theta^{3}\frac{|x-x_{0}|^{2}}{k}\nu^{2}dtdadx.
\end{eqnarray}
Taking $\epsilon$ small enough and $s$ quite large, we conclude that
\begin{eqnarray}\label{S1-9}
\nonumber&&
\left|\frac{s}{2}\int_{Q}(\Theta_{aa}+\Theta_{tt})\psi\nu^{2}dtdadx+s\int_{Q}\Theta_{ta}\psi\nu^{2}dtdadx\right| \\ &&
\leq\frac{sc_{1}}{4}\int_{Q}\Theta k \nu_{x}^{2}dtdadx+\frac{s^{3}c_{1}^{3}}{4}\int_{Q}\Theta^{3}\frac{|x-x_{0}|^{2}}{k}\nu^{2}dtdadx.
\end{eqnarray}
Taking into account the relations \eqref{S1-1} and \eqref{S1-2} we arrive to
\begin{eqnarray}\label{S1-10}
S_{1}\geq K_{1}s^{3}\int_{Q}\Theta^{3}\frac{|x-x_{0}|^{2}}{k}\nu^{2}dtdadx+K_{2}s\int_{Q}\Theta k \nu_{x}^{2}dtdadx
\end{eqnarray}
Hence,
\begin{eqnarray}\label{S1-11}
\nonumber &&
2<L^{+}_{s}\nu,L^{-}_{s}\nu>\geq m\left(s^{3}\int_{Q}\Theta^{3}\frac{|x-x_{0}|^{2}}{k}\nu^{2}dtdadx+s\int_{Q}\Theta k \nu_{x}^{2}dtdadx\right)\\ &&-2sc_{1}\int_{0}^{A}\int_{0}^{T}[k\Theta e^{2s\varphi} (x-x_{0})\nu_{x}^{2}]_{x=0}^{x=1} dtda
\end{eqnarray}
This steers to the following Carleman estimate verified by $\nu$ solution of \eqref{Carl-1}
\begin{eqnarray}\label{S1-12}
\nonumber &&s^{3}\int_{Q}\Theta^{3}\frac{|x-x_{0}|^{2}}{k}\nu^{2}dtdadx+s\int_{Q}\Theta k \nu_{x}^{2}dtdadx \\ &&\leq C_{6}\left(\int_{Q}h^{2}e^{2s\varphi}dtdadx+s\int_{0}^{A}\int_{0}^{T}[k\Theta e^{2s\varphi} (x-x_{0})\nu_{x}^{2}]_{x=0}^{x=1} dtda\right)
\end{eqnarray}
By the definition of $\nu$ we infer that
\begin{eqnarray}\label{S1-13}
\nu_{x}=s\varphi_{x}e^{s\varphi}w+e^{s\varphi}w_{x}, \hspace{0.25cm}e^{2s\varphi}w_{x}^{2}\leq2(\nu_{x}^{2}+s^{2}\varphi_{x}^{2}\nu^{2}).
\end{eqnarray}
Finally, the Carleman estimate \eqref{Inter-Carl} of \eqref{inter-adjoint sys-1} is obtained.
\\Now, If we apply the same inequality of Hardy-Poincar\'{e} in a similar way as before to the function $\nu:=e^{s\varphi}w$, taking into account the hypothesis on $\mu$ assumed in \eqref{hyp-beta-mu}, using the Carleman type inequality \eqref{Inter-Carl} for the function $h+\mu w$ and taking $s$ quite large we achieve the Proposition \ref{Intermediate-Carleman estimate}.
\endproof
 With the aid of the estimate \eqref{Inter-Carl} and Caccioppoli's inequality  \eqref{Caccio},  we can now show a $\omega$-local Carleman estimate for the system \eqref{inter-adjoint sys-2}. This result will be useful to show our main Carleman estimate replacing the second term $h$ by $-\beta (t,a,x)w(t,0,x)$. To this end, we introduce the following weight functions:
 \begin{equation}\label{poids}
\left\{
\begin{array}{l}
\Phi(t,a, x):=\Theta  (t, a)\Psi(x),
\\
\Psi(x)=e^{\kappa\sigma(x)}-e^{2\kappa\|\sigma\|_{\infty}},
\end{array}
\right.
\end{equation}
where $\Theta$ is given by \eqref{weight-functions}, $\kappa>0$ and $\sigma$ is the  function given by
\begin{equation}\label{sigma-2}
\left\{
\begin{array}{l}
\sigma \in C^{2}([0,1]),
\sigma(x)>0 \text{ in } (0,1), \sigma(0)=\sigma(1)=0, \\
\sigma_{x}(x)\neq0 \text{ in } [0,1]\backslash \omega_{0},
\end{array}
\right.
\end{equation}
where $ \omega_{0}\Subset\omega $ is an open subset.
The existence of the function $\sigma$ is proved in \cite{Fursikov}.
\\ On the other hand by the definition of $\varphi$ \eqref{weight-functions} and taking
\begin{eqnarray}\label{c-1-restr-0}
c_{1}\geq \max\left(\frac{k(1)(2-\gamma)(e^{2\kappa\|\sigma\|_{\infty}}-1)}{c_{2}k(1)(2-\gamma)-(1-x_{0})^{2}}
,\frac{k(0)(2-\gamma)(e^{2\kappa\|\sigma\|_{\infty}}-1)}{c_{2}k(0)(2-\gamma)-x_{0}^{2}}\right),
\end{eqnarray}
one can prove that
\begin{equation}\label{c-1-restr}
\varphi\leq \Phi.
\end{equation}
Our theorem is stated as follows
\begin{theorem}\label{Carle-estimate}
Assume that the assumptions
\eqref{hyp-k}, \eqref{hyp-beta-mu} and \eqref{H-P} hold. Let $A>0$ and $T>0$ be given. Then, there exist positive constants $C$ and $s_{0}$ such that for all $s\geq s_{0}$,  every solution $w$ of \eqref{inter-adjoint sys-2} satisfies
\begin{align}\label{Carle-esti}
&&\int_{Q}(s\Theta k w_{x}^{2}&+s^{3}\Theta^{3}\frac{(x-x_{0})^{2}}{k} w^{2})e^{2s\varphi} dtdadx
\leq C\left(\int_{Q}h^{2}e^{2s\Phi} dtdadx+ \int_{q}s^{3}\Theta^{3}w^{2}e^{2s\Phi}dtdadx\right),
\end{align}
\end{theorem}
To prove this theorem, we need the following result which represents the Carleman estimate of nondegenerate population dynamics systems. This inequality is stated as follows
\begin{proposition}\label{Carl-nondegenerate}
Let us consider the following system
\begin{eqnarray}\label{564}
{{\partial z} \over {\partial t}} + {{\partial z} \over {\partial a}}+(k(x)z_{x})_{x}-c(t, a, x)z = h && \text{ in } Q_{b},\\
\nonumber  z(t, a, b_{1})=z(t, a, b_{2})=0 && \text{ on } (0,T)\times(0,A),
\end{eqnarray}
where $Q_{b}:=(0,T)\times(0,A)\times(b_{1},b_{2})$, $(b_{1},b_{2})\subset[0,x_{0})$, or $(b_{1},b_{2})\subset(x_{0},1],$  $h \in L^{2}(Q_{b})$, $k \in C^{1}([0,1])$ is a strictly positive function and $c\in L^{\infty}(Q_{b})$.Then, there exist two positive constants $C$ and $s_{0}$, such that for any $s\geq s_{0}$, $z$ verifies the following
estimate
\begin{eqnarray}\label{570}
\nonumber & &
\int_{Q_{b}}(s^{3}\phi^{3}z^{2}+s\phi z_{x}^{2})e^{2s\Phi}dtdadx \\ & &
\leq C \left(\int_{Q_{b}}h^{2}e^{2s\Phi}dtdadx+\int_{\omega}\int_{0}^{A}\int_{0}^{T}s^{3}\phi^{3}z^{2}e^{2s\Phi}dtdadx\right),
\end{eqnarray}
where
\begin{eqnarray}\label{phi}
\phi(t,a,x)=\Theta(t,a)e^{\kappa\sigma(x)},
\end{eqnarray}
$\Theta$ and $\Phi$ are defined by \eqref{poids} and $\sigma$ by \eqref{sigma-2}.
\end{proposition}
Before giving the proof of Theorem \ref{Carle-estimate}, we note that a similar result was demonstrated in \cite[Lemma 2.1]{Ain3} in the case when $k$ is a positive constant, for any dimension $n$ without the source term $h$ and with the weight function $\Theta(t,a)=\frac{1}{t(T-t)a}$. By careful computations, the same proof can be adapted to \eqref{570} where $k$ is a positive general nondegenerate coefficient, with our weight function $\Theta(t,a)=\frac{1}{t^{4}(T-t)^{4}a^{4}}$ and the source term $h$.
\proof
Let us introduce the smooth cut-off function $\xi:\mathbb{R}\rightarrow \mathbb{R}$ defined by
\begin{equation}\label{cut-off func-1}
\begin{cases}
0\leq\xi(x)\leq1,&x \in [0,1],\\
\xi(x)=1, &   x \in [\lambda_{1}, \lambda_{2}],\\
\xi(x)=0, &x \in [0,1]\backslash \omega,
\end{cases}
\end{equation}
where $\lambda_{1}=\frac{x_{1}+2x_{0}}{3}$ and $\lambda_{2}=\frac{x_{0}+2x_{2}}{3}$.\\ Let $w$ be the solution of \eqref{inter-adjoint sys-2} and define $v:=\xi w$. Then, $v$ satisfies the following system
\begin{eqnarray}\label{inter-adjoint sys-3}
 {{\partial v} \over {\partial t}} + {{\partial v} \over {\partial a}}+(k(x)v_{x})_{x}-\mu(t, a, x)v &=& \overline{h},\\
\nonumber  v(t, a, 1)=v(t, a, 0)&=&0, \\
\nonumber  v(T, a, x)&=&\xi w_{T}(a, x),\\
\nonumber  v(t, A, x)&=&0,
\end{eqnarray}
where $\overline{h}:=\xi h+(k(x)\xi_{x}w)_{x}+k(x)w_{x}\xi_{x}$. \\ Using Carleman estimate \eqref{Inter-Carl} and the definition of $\xi$, one has
\begin{align}\label{Carle-esti-1}
&&\int_{Q}(s\Theta k v_{x}^{2}&+s^{3}\Theta^{3}\frac{(x-x_{0})^{2}}{k} v^{2})e^{2s\varphi} dtdadx
\leq C\int_{Q}\overline{h}^{2}e^{2s\varphi} dtdadx.
\end{align}
On the other hand, using again the definition of $\xi$ we can check readily that
\begin{eqnarray}\label{equ-1}
\nonumber &&\int_{\lambda_{1}}^{\lambda_{2}}\int_{0}^{A}\int_{0}^{T}(s\Theta k v_{x}^{2}+s^{3}\Theta^{3}\frac{(x-x_{0})^{2}}{k} v^{2})e^{2s\varphi} dtdadx\\&=&
\int_{\lambda_{1}}^{\lambda_{2}}\int_{0}^{A}\int_{0}^{T}(s\Theta k w_{x}^{2}+s^{3}\Theta^{3}\frac{(x-x_{0})^{2}}{k} w^{2})e^{2s\varphi} dtdadx.
\end{eqnarray}
Therefore, combining \eqref{Carle-esti-1} and \eqref{equ-1} we have
\begin{eqnarray}\label{ineq-1}
&& \nonumber\int_{\lambda_{1}}^{\lambda_{2}}\int_{0}^{A}\int_{0}^{T}(s\Theta k w_{x}^{2}+s^{3}\Theta^{3}\frac{(x-x_{0})^{2}}{k} w^{2})e^{2s\varphi} dtdadx\\ && \leq C\int_{Q}\overline{h}^{2}e^{2s\varphi} dtdadx.
\end{eqnarray}
Hence by means of Caccioppoli's inequality \eqref{Caccio} and \eqref{ineq-1}, we conclude that
\begin{eqnarray}\label{ineq-2}
&& \nonumber\int_{\lambda_{1}}^{\lambda_{2}}\int_{0}^{A}\int_{0}^{T}(s\Theta k w_{x}^{2}+s^{3}\Theta^{3}\frac{(x-x_{0})^{2}}{k} w^{2})e^{2s\varphi} dtdadx\\ && \leq C\left(\int_{Q}h^{2}e^{2s\varphi}dtdadx+\int_{q} s^{2}\Theta^{2}w^{2}e^{2s\varphi} dtdadx\right),
\end{eqnarray}
where $\omega^{'}$ of Lemma \ref{Caccio-1} here is exactly $(x_{1}, \lambda_{1})\cup(\lambda_{2}, x_{2})$. \\
Now, let $z:=\eta w$, with $\eta$ is the smooth cut-off function defined by
\begin{equation}\label{cut-off func-2}
\begin{cases}
0\leq\eta(x)\leq1,&x \in [0,1],\\
\eta(x)=0, &   x \in [0, \frac{\lambda_{3}+2\lambda_{2}}{3}],\\
\eta(x)=1, &x \in [\lambda_{2}, 1],
\end{cases}
\end{equation}
where $\lambda_{3}=\frac{x_{2}+2x_{0}}{3}$. We can observe easily that $\lambda_{3}<\frac{\lambda_{3}+2\lambda_{2}}{3}<\lambda_{2}$. Then, $z$ satisfies the following population dynamics equation
\begin{eqnarray}\label{inter-adjoint sys-4}
 {{\partial z} \over {\partial t}} + {{\partial z} \over {\partial a}}+k(x)z_{xx}+k^{'}(x)z_{x}-\mu(t, a, x)z &=& \widetilde{h}, \hspace{0.5cm} \text{ in } (\lambda_{3}, 1)\\
 \nonumber z(t, a, 1)=z(t, a, \lambda_{3})&=&0,  \hspace{0.5cm} \text{ in }, (0,T)\times(0,A)
\end{eqnarray}
where $\widetilde{h}:=\eta h+(k(x)\eta_{x}w)_{x}+k(x)w_{x}\eta_{x}$.\\ By assumption on $k$, we have $k(x)>0, \hspace{0.25cm} x\in(\lambda_{3}, 1)$. Hence, \eqref{inter-adjoint sys-4} is a nondegenerate model. In this case, applying Proposition \ref{Carl-nondegenerate} to the function $\widetilde{h}$ with $b_{1}=\lambda_{3}$ and $b_{2}=1$ and using again Caccioppoli's inequality \eqref{Caccio}, we infer that
\begin{eqnarray}\label{431}
\nonumber& &\int_{\lambda_{3}}^{1}\int_{0}^{A}\int_{0}^{T}(s^{3}\phi^{3}z^{2}+s\phi z_{x}^{2})e^{2s\Phi}dtdadx\\ \nonumber& & \leq C\left(\int_{\lambda_{3}}^{1}\int_{0}^{A}\int_{0}^{T}(\eta h+(k\eta_{x}w)_{x}+k\eta_{x}w_{x})^{2}e^{2s\Phi} dtdadx+ \int_{\omega}\int_{0}^{A}\int_{0}^{T} s^{3}\Theta^{3}w^{2}e^{2s\Phi} dtdadx\right)
\\ \nonumber&& \leq \widetilde{C}\left(\int_{Q}h^{2}e^{2s\Phi}+((k\eta_{x}w)_{x}+k\eta_{x}w_{x})^{2}e^{2s\Phi} dtdadx+ \int_{\omega}\int_{0}^{A}\int_{0}^{T} s^{3}\Theta^{3}w^{2}e^{2s\Phi} dtdadx\right)
\\ \nonumber&& \leq \widetilde{C}(\int_{Q}h^{2}e^{2s\Phi}dtdadx
+\int_{\omega^{'}}\int_{0}^{A}\int_{0}^{T}(8(k\eta_{x})^{2}w_{x}^{2}+2((k\eta_{x})_{x})^{2}w^{2})e^{2s\Phi} dtdadx\\ \nonumber&&+ \int_{\omega}\int_{0}^{A}\int_{0}^{T}s^{3}\Theta^{3} w^{2}e^{2s\Phi} dtdadx)\\ \nonumber & & \leq \widetilde{C}_{1}\left(\int_{Q}h^{2}e^{2s\Phi}dtdadx
+\int_{\omega^{'}}\int_{0}^{A}\int_{0}^{T}(w_{x}^{2}+w^{2})e^{2s\Phi} dtdadx+ \int_{\omega}\int_{0}^{A}\int_{0}^{T}s^{3}\Theta^{3} w^{2}e^{2s\Phi} dtdadx\right)\\ & & \leq \widetilde{C}_{2}\left(\int_{Q}h^{2}e^{2s\Phi}dtdadx+ \int_{\omega}\int_{0}^{A}\int_{0}^{T} s^{3}\Theta^{3}w^{2}e^{2s\Phi} dtdadx\right),
\end{eqnarray}
with $\omega^{'}:=(\frac{\lambda_{3}+2\lambda_{2}}{3}, \lambda_{2})$.
By the restriction \eqref{c-1-restr} there exists $c_{3}>0$ such that,  for $ (t,a,x)\in [0,T]\times[0,A]\times[\lambda_{3},1]$, we have
\begin{equation*}
\Theta k(x)e^{2s\varphi}\leq c_{3}\phi e^{2s\Phi},\quad
\Theta^{3}\frac{(x-x_{0})^{2}}{k(x)}e^{2s\varphi}\leq c_{3}\phi^{3}e^{2s\Phi}.
\end{equation*}
Then,
\begin{eqnarray}\label{ineq-z-1}
&& \nonumber\int_{\lambda_{3}}^{1}\int_{0}^{A}\int_{0}^{T}(s\Theta k z_{x}^{2}+s^{3}\Theta^{3}\frac{(x-x_{0})^{2}}{k} z^{2})e^{2s\varphi} dtdadx\\ && \leq c_{3}\int_{\lambda_{3}}^{1}\int_{0}^{A}\int_{0}^{T}(s^{3}\phi^{3}z^{2}+s\phi z_{x}^{2})e^{2s\Phi}dtdadx.
\end{eqnarray}
This inequality together with \eqref{431} lead to
\begin{eqnarray}\label{ineq-z-2}
&& \nonumber\int_{\lambda_{3}}^{1}\int_{0}^{A}\int_{0}^{T}(s\Theta k z_{x}^{2}+s^{3}\Theta^{3}\frac{(x-x_{0})^{2}}{k} z^{2})e^{2s\varphi} dtdadx\\ && \leq \widetilde{c_{3}}\left(\int_{Q}h^{2}e^{2s\Phi}dtdadx+ \int_{\omega}\int_{0}^{A}\int_{0}^{T} s^{3}\Theta^{3}w^{2}e^{2s\Phi} dtdadx\right).
\end{eqnarray}
Taking into account the definition of $\eta$ \eqref{cut-off func-2}, we can say that
\begin{eqnarray}\label{z-3}
\nonumber &&\int_{\lambda_{2}}^{1}\int_{0}^{A}\int_{0}^{T}(s\Theta k z_{x}^{2}+s^{3}\Theta^{3}\frac{(x-x_{0})^{2}}{k} z^{2})e^{2s\varphi} dtdadx\\&=&
\int_{\lambda_{2}}^{1}\int_{0}^{A}\int_{0}^{T}(s\Theta k w_{x}^{2}+s^{3}\Theta^{3}\frac{(x-x_{0})^{2}}{k} w^{2})e^{2s\varphi} dtdadx.
\end{eqnarray}
Hence,
\begin{eqnarray}\label{w-3}
\nonumber &&\int_{\lambda_{2}}^{1}\int_{0}^{A}\int_{0}^{T}(s\Theta k w_{x}^{2}+s^{3}\Theta^{3}\frac{(x-x_{0})^{2}}{k} w^{2})e^{2s\varphi} dtdadx\\&& \leq \widetilde{c_{3}}\left(\int_{Q}h^{2}e^{2s\Phi}dtdadx+ \int_{\omega}\int_{0}^{A}\int_{0}^{T} s^{3}\Theta^{3}w^{2}e^{2s\Phi} dtdadx\right),
\end{eqnarray}
as a consequence of \eqref{ineq-z-2} and \eqref{z-3}.
Arguing in the same way for $(0, \lambda_{1})$, one can show that
\begin{eqnarray}\label{w-4}
\nonumber &&\int_{0}^{\lambda_{1}}\int_{0}^{A}\int_{0}^{T}(s\Theta k w_{x}^{2}+s^{3}\Theta^{3}\frac{(x-x_{0})^{2}}{k} w^{2})e^{2s\varphi} dtdadx\\&& \leq \widetilde{c_{3}}\left(\int_{Q}h^{2}e^{2s\Phi}dtdadx+ \int_{\omega}\int_{0}^{A}\int_{0}^{T} s^{3}\Theta^{3}w^{2}e^{2s\Phi} dtdadx\right),
\end{eqnarray}
Finally, summing the inequalities \eqref{ineq-2}, \eqref{w-3} and \eqref{w-4} side by side, taking $s$ quite large and using again the restriction on $c_{1}$ \eqref{c-1-restr} we arrive to
 \begin{eqnarray}\label{w-5}
\nonumber &&\int_{0}^{1}\int_{0}^{A}\int_{0}^{T}(s\Theta k w_{x}^{2}+s^{3}\Theta^{3}\frac{(x-x_{0})^{2}}{k} w^{2})e^{2s\varphi} dtdadx\\&& \leq \widetilde{c_{3}}\left(\int_{Q}h^{2}e^{2s\Phi}dtdadx+ \int_{\omega}\int_{0}^{A}\int_{0}^{T} s^{3}\Theta^{3}w^{2}e^{2s\Phi} dtdadx\right),
\end{eqnarray}
and this is exactly the desired estimate \eqref{Carle-esti}.
\endproof
Before to provide the main Carleman estimate, we make the following remarks:
\begin{remark}
1/ The proof of our distributed-Carleman estimate \eqref{Carle-esti} is based on the cut-off functions and given by two different weighted functions $\varphi$ and $\Phi$, in addition by \eqref{c-1-restr} there is no positive constant $C$ such that:
$$e^{2s\Phi}\leq Ce^{2s\varphi}.$$
2/ Our proof is not based on the reflection method used for the proof of \cite[Lemma 4.1]{genni} which is needed to eliminate the boundary term arising in the classical Carleman estimate for nondegenerate heat equation.
\end{remark}
By the Carleman estimate \eqref{Carle-esti}, we are able to show the following $\omega$-Carleman estimate for the full adjoint system \eqref{adj-sys}
\begin{theorem}
Assume that the assumptions
\eqref{hyp-k}, \eqref{hyp-beta-mu} and \eqref{H-P} hold. Let $A>0$ and $T>0$ be given such that $0<T<\delta$, where $\delta\in(0,A)$ small enough. Then, there exist positive constants $C$ and $s_{0}$ such that for all $s\geq s_{0}$,  every solution $w$ of \eqref{inter-adjoint sys-2} satisfies
\begin{align}\label{main-Carle}
 && \int_{Q}(s\Theta k w_{x}^{2}+s^{3}\Theta^{3}\frac{(x-x_{0})^{2}}{k} w^{2})e^{2s\varphi} dtdadx\leq C\left( \int_{q}s^{3}\Theta^{3}w^{2}e^{2s\Phi}dtdadx+\int_{0}^{1}\int_{0}^{\delta}w_{T}^{2}(a,x)dadx\right),
\end{align}
for all $s\geq s_{0}$ and $\delta$ verifying \eqref{hyp-beta-mu}.
\end{theorem}
\proof
Applying  the inequality \eqref{Carle-esti} to  the function $h(t,a,x)=-\beta(t,a,x)w(t,0,x)$,  we have the existence of two positive constants $C$ and $s_{0}$ such that, for all $s\geq s_{0}$, the following inequality holds
\begin{eqnarray}
\nonumber
& &
s^{3}\int_{Q}  \Theta^{3}\frac{(x-x_{0})^{2}}{k(x)}w^{2}e^{2s\varphi} dtdadx + s\int_{Q}  \Theta k(x)w_{x}^{2}e^{2s\varphi} dtdadx \\ \nonumber& & \leq C
\left(\int_{Q}\beta^{2} w^{2}(t,0,x)e^{2s\Phi}dtdadx+\int_{q}s^{3}\Theta^{3}w^{2}e^{2s\Phi}dtdadx\right) \label{434}
\\  && \leq C\left(A\|\beta\|_{\infty}^{2}\int_{0}^{1}\int_{0}^{T} w^{2}(t,0,x)dtdx+\int_{q}s^{3}\Theta^{3}w^{2}e^{2s\Phi}dtdadx\right),
\end{eqnarray}
using \eqref{hyp-beta-mu}.
On the other hand, integrating over the characteristics lines and after a careful calculus we obtain the following implicit formula of $w$ solution of \eqref{adj-sys}
\begin{equation}\label{453}
\begin{cases}
w(t,a,\cdot)=\int_{0}^{A-a}S(A-a-l)\beta(t,A-l,\cdot)w(t,0,\cdot)dl, & \text{ if } a>t+(A-T)\\
w(t,a,\cdot)=S(T-t)w_{T}(T+(a-t),\cdot)+\int_{t}^{T}S(l-t)\beta(l,a,\cdot)w(l,0,\cdot)dl, & \text{ if } a\leq t+(A-T),
\end{cases}
\end{equation}
where $(S(t))_{t\geq 0}$ is the semi-group generated by the operator $A_{2}w=(kw_{x})_{x}-\mu w$.\\ Thus,
\begin{eqnarray}\label{w}
w(t,0,\cdot)=S(T-t)w_{T}(T-t,\cdot),
\end{eqnarray}
using the last hypothesis in \eqref{hyp-beta-mu} on $\beta$ . Injecting this formula in \eqref{434} and using the fact $(S(t))_{t\geq 0}$ is a bounded semi-group, we get
\begin{eqnarray}
\nonumber
& &
s^{3}\int_{Q}  \Theta^{3}\frac{(x-x_{0})^{2}}{k(x)}w^{2}e^{2s\varphi} dtdadx + s\int_{Q}  \Theta k(x)w_{x}^{2}e^{2s\varphi} dtdadx  \label{Carlesti}\\ \nonumber && \leq \widehat{C}\left(\int_{0}^{1}\int_{0}^{T} w_{T}^{2}(T-t,x)dtdx+\int_{q}s^{3}\Theta^{3}w^{2}e^{2s\Phi}dtdadx\right)\\ \nonumber && \leq \widehat{C}\left(\int_{0}^{1}\int_{0}^{T} w_{T}^{2}(m,x)dmdx+\int_{q}s^{3}\Theta^{3}w^{2}e^{2s\Phi}dtdadx\right)\\ && \leq \widehat{C}\left(\int_{0}^{1}\int_{0}^{\delta} w_{T}^{2}(m,x)dmdx+\int_{q}s^{3}\Theta^{3}w^{2}e^{2s\Phi}dtdadx\right),
\end{eqnarray}
since $T\in(0, \delta)$ with $\delta \in (0,A)$ small enough and this achieves the proof of \eqref{main-Carle}.
\endproof
\section{Observability inequality and null controllability results}\label{third-section}
\subsection{Observability inequality result}
The objective of this paragraph is to reach the observability inequality of the adjoint system \eqref{adj-sys}. To attain this purpose, we will combine the Carleman estimate \eqref{main-Carle} with the Hardy-Poincar\'{e} inequality stated in \cite[Proposition 2.3]{genni} and arguing in a similar way as in \cite{Ain3}. Our observability inequality is given by the following proposition
\begin{proposition}\label{prop-obser-ineq}
Assume that the assumptions \eqref{hyp-k}, \eqref{hyp-beta-mu} and \eqref{H-P} hold. Let $A>0$ and $T>0$ be given such that $0<T<\delta$, where $\delta\in(0,A)$ small enough. Then, there exists a positive constant $C_{\delta}$ such that  for every solution   $w$ of \eqref{adj-sys}, the  following observability inequality holds
 \begin{equation}\label{obser-ineq}
 \int_{0}^{1} \int_{0}^{A}w^{2}(0,a,x)dadx \leq C_{\delta} \left(\int_{q}w^{2} dtdadx+\int_{0}^{1}\int_{0}^{\delta}w_{T}^{2}(a,x)dadx\right).
\end{equation}
\end{proposition}
\begin{proof}
Let $w$ be a solution of \eqref{adj-sys}. Then for $\kappa>0$ to be defined later, $\widetilde{w}=e^{\kappa t}w$ is a
solution of
\begin{align}
&{{\partial \widetilde{w}} \over {\partial t}} + {{\partial \widetilde{w}} \over {\partial a}}+(k(x)\widetilde{w}_{x})_{x}-(\mu(t,a, x)+\kappa)\widetilde{w} =-\beta \widetilde{w}(t,0,x),\nonumber \\
&\widetilde{w}(t,a, 1)=\widetilde{w}(t,a, 0)=0,\label{obse-1-3}\\
&\widetilde{w}(T,a, x)=e^{\kappa T}w_{T}(a, x),\nonumber \\
&\widetilde{w}(t,A, x)=0.\nonumber
\end{align}
We point out that the parameter $\kappa$ considered here is not the same as in \eqref{poids}. Multiplying the first equation of \eqref{obse-1-3} by $\widetilde{w}$ and integrating
by parts on \\$Q_{t}=(0,t)\times(0,A)\times(0,1)$. Then, one obtains
\begin{eqnarray}\label{obs-1-4}
\nonumber &&-\frac{1}{2}\int_{Q_{A}}\widetilde{w}^{2}(t,a,x)dadx+\frac{1}{2}\int_{Q_{A}}w^{2}(0,a,x)dadx
+\frac{1}{2}\int_{0}^{1}\int_{0}^{t}\widetilde{w}^{2}(\tau,0,x)d\tau dx
\\ \nonumber &&+\kappa\int_{0}^{1}\int_{0}^{A}\int_{0}^{t}\widetilde{w}^{2}(\tau,a,x)d\tau dadx \leq
\frac{\|\beta\|_{\infty}^{2}}{4\epsilon^{'}}\int_{0}^{1}\int_{0}^{A}\int_{0}^{t}\widetilde{w}^{2}(\tau,a,x)d\tau dadx
\\ &&+\epsilon^{'}A\int_{0}^{1}\int_{0}^{t}\widetilde{w}^{2}(\tau,0,x)d\tau dx.
\end{eqnarray}
Thus, for $\kappa=\frac{\|\beta\|_{\infty}^{2}}{4\epsilon^{'}}$ and $\epsilon^{'}<\frac{1}{2A}$, one gets after integration over $(\frac{T}{4}, \frac{3T}{4})$
\begin{eqnarray}\label{obs-1-5}
\int_{Q_{A}}w^{2}(0,a,x)dadx\leq C_{12}e^{2\kappa T}\int_{Q_{A}}\int_{\frac{T}{4}}^{\frac{3T}{4}}w^{2}(t,a,x)dadx.
\end{eqnarray}
On the other hand, let us prove that there exists a positive constant $C_{\delta}$ such that
\begin{eqnarray}\label{obs1}
\int_{0}^{1}\int_{0}^{\delta-\frac{3T}{4}}\int_{\frac{T}{4}}^{\frac{3T}{4}} w^{2}(t,a,x)dtdadx\leq C_{\delta}\int_{0}^{1}\int_{0}^{\delta} w_{T}^{2}(a,x)dadx.
\end{eqnarray}
For this purpose, we will use the implicit formula of $w$ defined by \eqref{453} and we shall discuss the two cases, namely the case when $a>t+(A-T)$ and when $a\leq t+(A-T)$. In fact, if $a>t+(A-T)$ one has
\begin{eqnarray}\label{W}
\nonumber w(t,a,\cdot)&=&\int_{0}^{A-a}S(A-a-l)\beta(t,A-l,\cdot)w(t,0,\cdot)dl
\\ \nonumber &=& \int_{0}^{A-a}S(A-a-l)\beta(t,A-l,\cdot)S(T-t)w_{T}(T-t,\cdot)dl,
\end{eqnarray}
using \eqref{w}. Since $(S(t))_{t\geq 0}$ is a bounded semi-group and $\beta \in L^{\infty}(Q)$, one can see that for $T\in(0, \delta)$
\begin{eqnarray}\label{W1}
\nonumber && \int_{0}^{1}\int_{0}^{\delta-\frac{3T}{4}}\int_{\frac{T}{4}}^{\frac{3T}{4}} w^{2}(t,a,x)dtdadx
\\ \nonumber &&\leq \widetilde{C}_{10}\int_{0}^{1}\int_{\frac{T}{4}}^{\frac{3T}{4}}w_{T}^{2}(T-t,x)dtdx
\\ &&\leq \widetilde{C}_{10}\int_{0}^{1}\int_{0}^{\delta}w_{T}^{2}(m,x)dmdx,
\end{eqnarray}
Now, if $a\leq t+(A-T)$ one has
\begin{eqnarray} \label{W2}
\nonumber w(t,a,\cdot)&=&S(T-t)w_{T}(T+(a-t),\cdot)+\int_{t}^{T}S(l-t)\beta(l,a,\cdot)w(l,0,\cdot)dl
\\ \nonumber &=&S(T-t)w_{T}(T+(a-t),\cdot)+\int_{t}^{T}S(l-t)\beta(l,a,\cdot)S(T-l)w_{T}(T-l,\cdot)dl.
\end{eqnarray}
Thanks to the same argument employed to get \eqref{W1}, we conclude that
\begin{eqnarray}\label{W3}
\nonumber && \int_{0}^{1}\int_{0}^{\delta-\frac{3T}{4}}\int_{\frac{T}{4}}^{\frac{3T}{4}} w^{2}(t,a,x)dtdadx
\\ \nonumber && \leq 2\widetilde{C}_{11}(\int_{0}^{1}\int_{0}^{\delta-\frac{3T}{4}}\int_{\frac{T}{4}}^{\frac{3T}{4}}w_{T}^{2}(T+(a-t),x)dtdadx
\\ &&+\int_{0}^{1}\int_{t}^{T}w_{T}^{2}(T-l,x)dldx).
\end{eqnarray}
On one hand, we can check that
\begin{eqnarray}\label{W4}
\nonumber && \int_{0}^{1}\int_{0}^{\delta-\frac{3T}{4}}\int_{\frac{T}{4}}^{\frac{3T}{4}}w_{T}^{2}(T+(a-t),x)dtdadx
\\ \nonumber && \leq \int_{0}^{1}\int_{0}^{\delta-\frac{3T}{4}}\int_{\frac{T}{4}}^{\frac{3T}{4}}w_{T}^{2}(a+m,x)dmdadx
\\ \nonumber && \leq \int_{0}^{1}\int_{0}^{\delta-\frac{3T}{4}}\int_{a+\frac{T}{4}}^{a+\frac{3T}{4}}w_{T}^{2}(z,x)dzdadx
\\ \nonumber && \leq \int_{0}^{1}\int_{0}^{\delta-\frac{3T}{4}}\int_{0}^{\delta}w_{T}^{2}(z,x)dzdadx
\\ && \leq \delta\int_{0}^{1}\int_{0}^{\delta}w_{T}^{2}(z,x)dzdx.
\end{eqnarray}
On the other hand, we have the following inequality
\begin{eqnarray}\label{W5}
\nonumber && \int_{0}^{1}\int_{t}^{T}w_{T}^{2}(T-l,x)dldx=\int_{0}^{1}\int_{0}^{T-t}w_{T}^{2}(z,x)dzdx
\\ && \leq\int_{0}^{1}\int_{0}^{\delta}w_{T}^{2}(z,x)dzdx.
\end{eqnarray}
Combining the inequalities \eqref{W3}, \eqref{W4} and \eqref{W5} we get
\begin{eqnarray}\label{W6}
\nonumber && \int_{0}^{1}\int_{0}^{\delta-\frac{3T}{4}}\int_{\frac{T}{4}}^{\frac{3T}{4}} w^{2}(t,a,x)dtdadx
\\ \nonumber && \leq\widetilde{C}_{12}\int_{0}^{1}\int_{0}^{\delta}w_{T}^{2}(z,x)dzdx.
\end{eqnarray}
Subsequently, \eqref{obs1} occurs in both studied cases. Therefore, in the light of inequality \eqref{obs-1-5} we conclude that
\begin{eqnarray}\label{obs-1-6}
\nonumber &&\int_{Q_{A}}w^{2}(0,a,x)dadx\leq \widetilde{C}_{13}\int_{0}^{1}\int_{0}^{\delta}w_{T}^{2}(a,x)dadx
\\ && +\frac{2e^{2\kappa T}}{T}\int_{0}^{1}\int_{\delta-\frac{3T}{4}}^{A}\int_{\frac{T}{4}}^{\frac{3T}{4}}w^{2}(t,a,x)dtdadx.
\end{eqnarray}
Now, let $p$ defined by \eqref{p}. Then, using the hypotheses \eqref{hyp-k} on $k$  the function $x\mapsto\frac{(x-x_{0})^{2}}{p(x)}$ is nonincreasing in the left of $x_{0}$ and nondecreasing in the right of $x_{0}$.
Hence, applying Hardy-Poincar\'{e} inequality (see \cite[Proposition 2.3]{genni}) and taking into account the definition of $\varphi$ stated in \eqref{weight-functions} we have
\begin{eqnarray}\label{obs-1-7}
\nonumber &&\int_{Q_{A}}w^{2}(0,a,x)dadx\leq
\widetilde{C}_{13}\int_{0}^{1}\int_{0}^{\delta}w_{T}^{2}(a,x)dadx
\\ && +C_{\delta}^{13}\int_{0}^{1}\int_{\delta-\frac{3T}{4}}^{A}\int_{\frac{T}{4}}^{\frac{3T}{4}}s\Theta k(x) w_{x}^{2}(t,a,x)e^{2s\varphi}dtdadx.
\end{eqnarray}
Therefore, using Carleman estimate \eqref{main-Carle} we infer
\begin{eqnarray}\label{obs-1-8}
\nonumber && \int_{Q_{A}}w^{2}(0,a,x)dadx\leq
\widetilde{C}_{\delta}^{15}\left(\int_{q}s^{3}\Theta^{3}w^{2}e^{2s\Phi}dtdadx
+\int_{0}^{1}\int_{0}^{\delta}w_{T}^{2}(a,x)dadx\right),
\end{eqnarray}
and then the proof is finished using the fact that $\sup_{(t,a,x)\in Q}s^{d}\Theta^{d}e^{2s\Phi}<+\infty \hspace{0.25cm} , \,\forall d\in\mathbb{R}$.
\end{proof}
\subsection{Null controllability result}
In the previous paragraph, we obtained the observability inequality of system \eqref{adj-sys}. Such a tool will be very useful to prove the null controllability of the model \eqref{405} in the case where $T\in(0,\delta)$ as we emphasized in the introduction. Our main result is provided in the following theorem
\begin{theorem}\label{null-contr-result-1}
Assume that the dispersion coefficient $k$ satisfies \eqref{hyp-k} and the natural rates $\beta$ and $\mu$ verify \eqref{hyp-beta-mu}. Let $A, T>0$ be given such that $0<T<\delta$, where $\delta\in(0,A)$ small enough. For all $y_{0} \in L^{2}(Q_{A})$, there exists a control $\vartheta \in L^{2}(q)$ such that the associated solution of \eqref{405}  verifies
\begin{equation}\label{406}
y(T,a, x)=0, \quad  \text{ a.e. in } (\delta, A)\times(0,1).
\end{equation}
Furthermore, there exists a positive constant $C_{10}$ which depends on $\delta$ such that $\vartheta$ satisfies the following inequality.
\begin{equation}\label{control-estimate}
\int_{q}\vartheta^{2}(t,a,x) dtdadx \leq C_{10}\int_{Q_{A}}y_{0}^{2}(a,x)dadx.
\end{equation}
 $C_{10}$ is called the control cost.
\end{theorem}
Before proceeding to the proof of Theorem \ref{null-contr-result-1}, we shall make the following remark:
\begin{remark}\label{control-remark}
The inequality \eqref{control-estimate} shows us clearly that the control that we are looking for depends on $\delta$ and the initial distribution $y_{0}$.
 \end{remark}
\proof
Let $ \varepsilon>0$ and consider the following cost function
 $$J_{\varepsilon}(\vartheta)=\frac{1}{2\varepsilon}\int_{0}^{1}\int_{\delta}^{A}y^{2}(T,a, x)dadx
+\frac{1}{2}\int_{q}\vartheta^{2}(t,a,x) dtdadx.$$
We can prove that $J_{\varepsilon}$ is continuous, convex and coercive. Then, it admits at least one minimizer $\vartheta_{\varepsilon}$ and we have
\begin{eqnarray}\label{438}
\vartheta_{\varepsilon}=-w_{\varepsilon}(t,a,x)\chi_{\omega}(x) \hspace{0.25cm} \text{ in } Q,
\end{eqnarray}
with $w_{\varepsilon}$ is the solution of the following system
\begin{eqnarray}\label{439}
{{\partial w_{\varepsilon}} \over {\partial t}} + {{\partial w_{\varepsilon}} \over {\partial a}}+(k(x)(w_{\varepsilon})_{x})_{x}-\mu(t,a, x)w_{\varepsilon} =-\beta w_{\varepsilon}(t,0,x) && \text{ in } Q,\\
\nonumber  w_{\varepsilon}(t,a, 1)=w_{\varepsilon}(t,a, 0)=0  && \text{ on }(0,T)\times (0,A),\\
\nonumber  w_{\varepsilon}(T,a, x)=\frac{1}{\varepsilon}y_{\varepsilon}(T,a, x)\chi_{(\delta, A)}(a) && \text{ in }Q_{A},\\
\nonumber  w_{\varepsilon}(t,A, x)=0 && \text{ in } Q_{T},
\end{eqnarray}
and $y_{\varepsilon}$ is the solution of the system \eqref{405} associated to the control $\vartheta_{\varepsilon}$. Multiplying \eqref{439} by $y_{\varepsilon}$, integrating over $Q$, using \eqref{438} and the Young inequality we obtain
\begin{eqnarray}\label{440}
\nonumber & &
\frac{1}{\varepsilon}\int_{0}^{1}\int_{\delta}^{A}y_{\varepsilon}^{2}(T,a, x)dadx+\int_{q}\vartheta^{2}_{\varepsilon}(t,a,x) dtdadx\\ \nonumber&=&\int_{Q_{A}}y_{0}(a,x)w_{\varepsilon}( 0, a, x)dadx \\ \nonumber& & \leq \frac{1}{4C_{\delta}}\int_{Q_{A}}w_{\varepsilon}^{2}( 0, a, x)dadx +C_{\delta}\int_{Q_{A}}y_{0}^{2}(a,x)dadx,
\end{eqnarray}
with $C_{\delta}$ is the constant of the observability inequality  \eqref{obser-ineq}. This again leads to
\begin{eqnarray}\label{441}
\nonumber & &\frac{1}{\varepsilon}\int_{0}^{1}\int_{\delta}^{A}y_{\varepsilon}^{2}(T,a, x)dadx+\int_{q}\vartheta^{2}_{\varepsilon}(t,a,x) dtdadx \leq \frac{1}{4}\int_{q}w^{2}dtdadx +C_{\delta}\int_{Q_{A}}y_{0}^{2}(a,x)dadx.
\end{eqnarray}
Keeping in the mind \eqref{438}, we conclude that
\begin{eqnarray}
\frac{1}{\varepsilon}\int_{0}^{1}\int_{\delta}^{A}y_{\varepsilon}^{2}(T,a, x)dadx+\frac{3}{4}\int_{q}\vartheta^{2}_{\varepsilon}(t,a,x) dtdadx \leq C_{\delta}\int_{Q_{A}}y_{0}^{2}(a,x)dadx.
\end{eqnarray}
Hence, it follows that
\begin{equation}\label{442}
\begin{cases}
\int_{0}^{1}\int_{\delta}^{A}y_{\varepsilon}^{2}(T,a, x)dadx \leq \varepsilon C_{\delta}\int_{Q_{A}}y_{0}^{2}(a,x)dadx\\
\int_{q}\vartheta^{2}_{\varepsilon}(t,a,x) dtdadx \leq \frac{4C_{\delta}}{3}\int_{Q_{A}}y_{0}^{2}(a,x)dadx.
\end{cases}
\end{equation}
 Then, we can extract two subsequences of $y_{\varepsilon}$ and $\vartheta_{\varepsilon}$ denoted also by $\vartheta_{\varepsilon}$
 and $y_{\varepsilon}$ that converge weakly towards $\vartheta$ and $y$ in $L^{2}(q)$ and $L^{2}((0, T)\times(0, A); H^{1}_{k}(0, 1))$ respectively. Now, by a variational technic,
  we prove that $y$ is a solution of  \eqref{405}  corresponding to the control  $\vartheta$ and, by the first estimate of \eqref{442}, $y$ satisfies \eqref{406} for $T\in(0,\delta)$ and this shows our claimed Theorem\ref{null-contr-result-1}
\endproof

\section{Appendix}\label{fourth-section}
As we said in the introduction, this Appendix is concerned with a result which plays an important role to show the $\omega$-Carleman estimate associated to the full adjoint system \eqref{adj-sys} namely the Caccioppoli's inequality which is stated in the following lemma
\begin{lemma}\label{Caccio-1}Let $\omega^{'}\subset\subset\omega$ and $w$ be the solution of \eqref{inter-adjoint sys-2}. Suppose that $x_{0}\notin \overline{\omega^{'}}$. Then, there exists a positive constant $C$ such that $w$ verifies
\begin{equation}\label{Caccio}
\int_{\omega^{'}}\int_{0}^{A}\int_{0}^{T}w_{x}^{2}e^{2s\varphi} dtdadx\leq C\left(\int_{q}s^{2}\Theta^{2}w^{2}e^{2s\varphi} dtdadx+\int_{q}h^{2}
e^{2s\varphi}dtdadx\right).
\end{equation}
\end{lemma}
\begin{proof} Define the following smooth cut-off function $\zeta:\mathbb{R}\rightarrow \mathbb{R}$
\begin{equation}\label{81}
\left\{
\begin{array}{l}
0\leq\zeta(x)\leq1, \hspace{0.5cm} x \in \mathbb{R},\\
\zeta(x)=0, \hspace{0.5cm}   x<x_1 \text{ and } x>x_2,\\
\zeta(x)=1, \hspace{0.5cm}   x \in \omega^{'}.\\
\end{array}
\right.
\end{equation}
For the solution  $w$ of \eqref{inter-adjoint sys-2}, we have
\begin{eqnarray}\label{131}
\nonumber & & 0=\int_{0}^{T}\frac{d}{dt}\left[\int_{0}^{1}\int_{0}^{A}
\zeta^{2}e^{2s\varphi}w^{2}dadx\right]dt
\\ \nonumber &=&2s \int_{0}^{1}\int_{0}^{A}
\int_{0}^{T}\zeta^{2}\varphi_{t}w^{2}e^{2s\varphi}dtdadx
+2\int_{0}^{1}\int_{0}^{A}
\int_{0}^{T}\zeta^{2}ww_{t}
e^{2s\varphi}dtdadx
\\ \nonumber &=&2s\int_{0}^{1}\int_{0}^{A}
\int_{0}^{T}\zeta^{2}\varphi_{t}w^{2}
e^{2s\varphi}dtdadx+
2\int_{0}^{1}\int_{0}^{A}\int_{0}^{T}\zeta^{2}w(-(kw_{x})_{x}-w_{a}+h+\mu w)e^{2s\varphi}dtdadx.
\end{eqnarray}
Then, integrating by parts we obtain
\begin{eqnarray*}\label{82}
\nonumber 2\int_{Q}k\zeta^{2}e^{2s\varphi}w_{x}^{2}dtdadx&=&-2s\int_{Q}\zeta^{2}w^{2}\psi(\Theta_{a}+
\Theta_{t})e^{2s\varphi}dtdadx-2\int_{Q}\zeta^{2}whe^{2s\varphi}dtdadx\\ &&-2\int_{Q}\zeta^{2}\mu w^{2}e^{2s\varphi}dtdadx
+\int_{Q}(k(\zeta^{2}e^{2s\varphi})_{x})_{x}w^{2}dtdadx.
\end{eqnarray*}
On the other hand, by the definitions of $\zeta$, $\psi$ and $\Theta$, thanks to Young inequality, taking $s$ quite large and using the fact that $x_{0}\notin \overline{\omega^{'}}$, one can prove the existence of a positive constant $c$ such that
\begin{eqnarray}\label{88}
\nonumber&&2\int_{Q}k\zeta^{2}e^{2s\varphi}w_{x}^{2}dtdadx\geq 2\min_{x\in\omega^{'}}k(x)\int_{\omega^{'}}\int_{0}^{A}\int_{0}^{T}w_{x}^{2}e^{2s\varphi}dtdadx,\\
\nonumber&&\int_{Q}(k(\zeta^{2}e^{2s\varphi})_{x})_{x}w^{2}dtdadx\leq c\int_{\omega}\int_{0}^{A}\int_{0}^{T}s^{2}\Theta^{2}w^{2}e^{2s\varphi}dtdadx,\\
\nonumber&&-2s\int_{Q}\zeta^{2}w^{2}\psi(\Theta_{a}+\Theta_{t})e^{2s\varphi}dtdadx\leq c\int_{\omega}\int_{0}^{A}\int_{0}^{T}s^{2}\Theta^{2}w^{2}e^{2s\varphi}dtdadx,\\
\nonumber&&-2\int_{Q}\zeta^{2}whe^{2s\varphi}dtdadx\leq c\left(\int_{\omega}\int_{0}^{A}\int_{0}^{T}s^{2}\Theta^{2}w^{2}e^{2s\varphi}dtdadx+\int_{\omega}\int_{0}^{A}\int_{0}^{T}h^{2}e^{2s\varphi}dtdadx\right),\\
\nonumber&&-2\int_{Q}\zeta^{2}\mu w^{2}e^{2s\varphi}dtdadx \leq c\int_{\omega}\int_{0}^{A}\int_{0}^{T}s^{2}\Theta^{2}w^{2}e^{2s\varphi}dtdadx.
\end{eqnarray}
This all together imply that there is $ C>0$ such that
\begin{eqnarray*}
& &\int_{\omega^{'}}\int_{0}^{A}\int_{0}^{T}w_{x}^{2} e^{2s\varphi}dtdadx\leq C\left(\int_{q}s^{2}\Theta^{2}w^{2}e^{2s\varphi} dtdadx+\int_{q}h^{2}
e^{2s\varphi}dtdadx\right).
\end{eqnarray*}
Thus, the proof is achieved.
\end{proof}
\begin{remark}\label{rem-cacc}
The Lemma \ref{Caccio-1} remains true for any function $\pi \in  C([0,1], (-\infty, 0))\cap C^{1}([0,1]\backslash \{x_{0}\}, (-\infty, 0))$ and verifying
\begin{eqnarray}
|\pi_{x}|\leq\frac{c}{\sqrt{k}}, \hspace{0.5cm} \text{ for } x\in[0,1]\backslash\{x_{0}\},
\end{eqnarray}
where $c>0$. see \cite[Proposition 4.2]{genni} for more details.
\end{remark}
\textbf{{\emph{Acknowledgements}}}
\\The authors would like to thank deeply the anonymous referee and the Professors B. Ainseba and L. Maniar for their fruitful and several remarks which allow us to realize this work.

\end{document}